\documentclass[11pt,a4paper]{article}
\RequirePackage{amsmath}
\RequirePackage{amssymb,latexsym}
\RequirePackage{amsthm}
\RequirePackage{enumerate}
\RequirePackage[hang,flushmargin,bottom,stable]{footmisc}
\RequirePackage[margin=0.35in,format=hang,font=small,labelfont=bf]{caption}
\RequirePackage[normalem]{ulem}
\RequirePackage{needspace}
\RequirePackage{graphicx,tikz}
\RequirePackage[T1]{fontenc}
\RequirePackage[sc]{mathpazo}
\RequirePackage[colorlinks=true,urlcolor=blue,linkcolor=blue,citecolor=blue]{hyperref}
\usepackage{multirow}
\usepackage{setspace}

\hypersetup{
pdfstartview={XYZ null null 1.00},
pdflang={en-GB},
}

\newtheorem{propO}{Proposition}

\newenvironment{bullets} {\vspace{-9pt}\begin{itemize}\itemsep0pt} {\end{itemize}\vspace{-9pt}}
\newenvironment{bulletnums} {\vspace{-9pt}\begin{enumerate}\itemsep0pt} {\end{enumerate}\vspace{-9pt}}

\newcommand{\ba}{\boldsymbol{a}}
\newcommand{\bv}{\boldsymbol{v}}
\newcommand{\ceil}[1]{\left\lceil #1 \right\rceil}
\newcommand{\eq}{\mathchoice{\;=\;}{=}{=}{=}}
\newcommand{\Erdos}{Erd\H{o}s}
\newcommand{\geqs}{\geqslant}
\newcommand{\leqs}{\leqslant}
\newcommand{\limsupinfty}[1][n]{\limsup\limits_{#1\rightarrow\infty}}
\newcommand{\pho}{\phantom{0}}
\newcommand{\tel}{\textellipsis}
\newcommand{\vdc}{Van der Corput} 
\newcommand{\vphi}{\varphi}

\newcommand{\myTitle}{On balancing consecutive slices of cake}
\title{\textbf{\myTitle}}

\author{$\phantom{{}^\dagger}$David Bevan${}^\dagger$}

\hypersetup{
pdftitle={\myTitle},
pdfauthor={David Bevan},
pdfstartview={XYZ null null 1.00}
}

\date{}

\begin{document}
\maketitle

{\begin{NoHyper}
\let\thefootnote\relax\footnotetext
{${}^\dagger$The University of Strathclyde, Glasgow, Scotland; email: dibevan@hotmail.co.uk.}
\end{NoHyper}}

{\begin{NoHyper}
\let\thefootnote\relax
\end{NoHyper}}

\begin{abstract}
\noindent
Let $\ba=(a_i)_{i=1}^\infty$ be an infinite sequence of points on a circle. 
The first $n$ of these points cuts the circle into $n$ pieces.
For any given~$r$, let $\mu^r_n(\ba)$ be the ratio between the maximum and minimum sizes of $r$ consecutive pieces.
Addressing a question of De~Bruijn and \Erdos{}, we define a family of sequences for which the asymptotic least upper bound of this ratio,
\[
\mu_r(\ba) \eq \limsupinfty\mu^r_n(\ba) ,
\]
can easily be calculated. Hence, for small $r$, we present upper bounds on $\inf\mu_r(\ba)$. 

\end{abstract}

\section{Introduction}

Suppose $\ba=(a_i)_{i=1}^\infty$ is an infinite sequence of points on a circle, or equivalently an infinite sequence of radial cuts of a circular cake.
The first $n$ of these cuts divides the cake into $n$ pieces, whose sizes are proportional to the angular separation between adjacent cuts.
For fixed~$r$, let us call $r$ consecutive pieces a \emph{portion},
and let $\mu^r_n(\ba)$ be the ratio between the maximum and minimum sizes of a portion.
Since this is the only parameter of interest, we simply refer to it as \emph{the ratio} for a given value of $r$.

Now let
\[\mu_r(\ba)
\eq\limsupinfty\mu^r_n(\ba) 
\]
denote the asymptotic least upper bound of the ratios.
Finally, let $\mu_r\eq\inf\mu_r(\ba)$, where the infimum is taken over all possible sequences of cuts.
Thus $\mu_r$ represents the asymptotically best possible balancing of the sizes of portions of $r$ consecutive pieces.

De~Bruijn and \Erdos{}~\cite{DeBE1949} prove that $\mu_1=2$ and that, for each $r$, we have $\mu_r\geqs1+1/r$.
They also conjecture that $r(\mu_r-1)$ is unbounded.

There has been notable recent interest in this question.
Firstly, Korsky~\cite{Korsky2026} improves the lower bound to $\mu_r\geqs1+r/(r^2-1)$.
He also conjectures that, in fact, $\mu_2=2$.
In a subsequent paper~\cite{BevanCakeSlicingTwo}, we disprove this conjecture: If $\rho\approx0.75488$ is the real root of $\rho^2+\rho^3=1$, then $\mu_2\leqs1+\rho$.

Secondly, Cl\'ement and Steinerberger~\cite{CS2025} establish, by analysis of the binary \vdc{} sequence, 
that there is an absolute constant~$c$ such that $\mu_r\leqs1+c\log r/r$ for each~$r>1$,
proving a conjecture of Brethouwer~\cite{Brethouwer2024}.
It is perhaps worth noting that it may be easier to establish the weaker conjecture that $r(\mu_r(\ba)-1)$ is unbounded for every sequence $\ba$.


Our goal is to establish upper bounds on $\mu_r$ for small values of~$r$. 
Our results are shown in Table~\ref{tblBounds} along with the upper bound for $\mu_2$ from~\cite{BevanCakeSlicingTwo} and Korsky's lower bounds for comparison, and in Table~\ref{tblResults} with details of the sequences of cuts that establish the bounds.

\begin{table}[ht]
\small
\[
\renewcommand*{\arraystretch}{1.2}
\begin{array}{c||c|c|c|c|c|c|c|c|c|}
r&2&3&4&5&6&7&8&9&10 \\\hline
\multirow{2}{*}{$\mu_r\leqs$}& 1+\rho & {3/2} & {3/2} & {7/5} & {4/3} & {5/4} & {5/4} & {11/9} & {6/5} \\
& 1.75488 & 1.5 & 1.5 & 1.4 & 1.33333 & 1.25 & 1.25 & 1.22222 & 1.2 \\\hline
\multirow{2}{*}{$\mu_r\geqs$}& {5/3} & {11/8} & {19/15} & {29/24} & {41/35} & {55/48} & {71/63} & {89/80} & {109/99} \\
& 1.66667 & 1.375 & 1.26667 & 1.20833 & 1.17143 & 1.14583 & 1.12698 & 1.1125 & 1.10101
\end{array}
\]
  \caption{Bounds on $\mu_r$ for small $r$}\label{tblBounds}
\end{table}

To establish our upper bounds, in Section~\ref{sectCutting} we define a family of ``periodic'' sequences of cuts.
In these sequences, following some initial cuts, the pieces always have at most two distinct sizes, the larger pieces being exactly twice the size of smaller ones.
At each step, we simply bisect a large piece, creating an adjacent pair of small pieces.
After bisecting every large piece, each piece has the same size, which we now consider to be the new size of a large piece.
These properties are also satisfied by the (binary) \vdc{} sequence, 
and we show in Section~\ref{sectVDC} that, for any given~$r$, the \vdc{} sequence is 
essentially the same as certain sequences in our family.

\section{Periodic cake cutting}\label{sectCutting}

Each of our cutting sequences is defined by a permutation $\sigma=\sigma_1\sigma_2\ldots\sigma_p$ of 
$\{1,\ldots,p\}$, 
for some \emph{period length}~$p$.
We call $\sigma$ the \emph{recipe} for the sequence.

Formally, the cutting process is as follows:

\begin{bulletnums}
\raggedright
  \item Initially, cut the cake into $p$ equally-sized pieces. Let $m=1$. 
  
  The variable $m$ will increase in value as the process continues.
  
  \item We consider the $m p$ pieces to consist of $m$ \emph{periods}, each consisting of $p$ consecutive pieces. 
  
  By suitable scaling, we may assume that each piece is a \emph{large} piece of size~2.
  
  In turn, for each $i=1,\ldots,p$: \\
  \quad In turn, for each $j=1,\ldots,m$: \\
  \quad\quad Bisect the $\sigma_i$-th of the (original) large pieces in the $j$-th period, creating an adjacent \\\quad\quad pair of \emph{small} pieces, each of size~1.
  
 \begin{bullets} 
  \item Note that after each $m$ cuts (bisecting the ``same'' large piece in each of the $m$ periods), the sizes of the pieces are periodic with period~$p$.
  \end{bullets}

  \item After bisecting every large piece, we have $2m p$ equally-sized pieces. We now repeat step~2 with $m$ doubled in value.
  
  This repeats \emph{ad infinitum}, with $m$ successively equal to each power of two.
\end{bulletnums}

\newcounter{x}
\newcommand{\splits}[4]  
{
  \footnotesize  
  \setcounter{x}{0} 
  \foreach \d in {#1}
  {
    \ifnum1=\d
    {\draw[fill=yellow!36] (\thex,-.5) rectangle (\thex+2,1.5);}
    \else
    {\draw[fill=yellow!12] (\thex,-.5) rectangle (\thex+4,1.5);}
    \fi
    \addtocounter{x}{\d}
    \node at (\thex,.5) {\d};
    \addtocounter{x}{\d}
  }
  \node at (60,.5) {#2\,:\,#3};
  \node at (-2,.5) {#4};
  \draw[very thick] (0,-.5)--(0,1.5);
  \draw[very thick] (28,-.5)--(28,1.5);
  \draw[very thick] (56,-.5)--(56,1.5);
}

\begin{figure}[t]
  \newcommand{\phoc}{\phantom{0:}}
  \centering
\begin{tikzpicture}[scale=.19]
\footnotesize
 \node at (-2,0) {\phoc};
 \node at (2,0) {1}; 
 \node at (6,0) {2}; 
 \node at (10,0) {3}; 
 \node at (14,0) {4}; 
 \node at (18,0) {5}; 
 \node at (22,0) {6}; 
 \node at (26,0) {7}; 
 \node at (30,0) {1}; 
 \node at (34,0) {2}; 
 \node at (38,0) {3}; 
 \node at (42,0) {4}; 
 \node at (46,0) {5}; 
 \node at (50,0) {6}; 
 \node at (54,0) {7}; 
 \node at (60,0) {\phantom{00\,:\,00}};
\end{tikzpicture}\\[2pt]  
  
\begin{tikzpicture}[scale=.19]\splits{2,2,2,2,2,2,2,2,2,2,2,2,2,2}{10}{10}{\phoc}\end{tikzpicture}\\[2pt]

\begin{tikzpicture}[scale=.19]\splits{1,1,2,2,2,2,2,2,2,2,2,2,2,2,2}{10}{8\pho}{1:}\end{tikzpicture}\\[2pt]

\begin{tikzpicture}[scale=.19]\splits{1,1,2,2,2,2,2,2,1,1,2,2,2,2,2,2}{10}{8\pho}{\phoc}\end{tikzpicture}\\[2pt]

\begin{tikzpicture}[scale=.19]\splits{1,1,2,2,1,1,2,2,2,1,1,2,2,2,2,2,2}{10}{7\pho}{4:}\end{tikzpicture}\\[2pt]

\begin{tikzpicture}[scale=.19]\splits{1,1,2,2,1,1,2,2,2,1,1,2,2,1,1,2,2,2}{\pho8}{7\pho}{\phoc}\end{tikzpicture}\\[2pt]

\begin{tikzpicture}[scale=.19]\splits{1,1,1,1,2,1,1,2,2,2,1,1,2,2,1,1,2,2,2}{\pho8}{6\pho}{2:}\end{tikzpicture}\\[2pt]

\begin{tikzpicture}[scale=.19]\splits{1,1,1,1,2,1,1,2,2,2,1,1,1,1,2,1,1,2,2,2}{\pho8}{6\pho}{\phoc}\end{tikzpicture}\\[2pt]

\begin{tikzpicture}[scale=.19]\splits{1,1,1,1,2,1,1,1,1,2,2,1,1,1,1,2,1,1,2,2,2}{\pho8}{6\pho}{5:}\end{tikzpicture}\\[2pt]

\begin{tikzpicture}[scale=.19]\splits{1,1,1,1,2,1,1,1,1,2,2,1,1,1,1,2,1,1,1,1,2,2}{\pho7}{6\pho}{\phoc}\end{tikzpicture}\\[2pt]

\begin{tikzpicture}[scale=.19]\splits{1,1,1,1,1,1,1,1,1,1,2,2,1,1,1,1,2,1,1,1,1,2,2}{\pho7}{5\pho}{3:}\end{tikzpicture}\\[2pt]

\begin{tikzpicture}[scale=.19]\splits{1,1,1,1,1,1,1,1,1,1,2,2,1,1,1,1,1,1,1,1,1,1,2,2}{\pho7}{5\pho}{\phoc}\end{tikzpicture}\\[2pt]

\begin{tikzpicture}[scale=.19]\splits{1,1,1,1,1,1,1,1,1,1,1,1,2,1,1,1,1,1,1,1,1,1,1,2,2}{\pho7}{5\pho}{6:}\end{tikzpicture}\\[2pt]

\begin{tikzpicture}[scale=.19]\splits{1,1,1,1,1,1,1,1,1,1,1,1,2,1,1,1,1,1,1,1,1,1,1,1,1,2}{\pho6}{5\pho}{\phoc}\end{tikzpicture}\\[2pt]

\begin{tikzpicture}[scale=.19]\splits{1,1,1,1,1,1,1,1,1,1,1,1,1,1,1,1,1,1,1,1,1,1,1,1,1,1,2}{\pho6}{5\pho}{7:}\end{tikzpicture}\\[2pt]

\begin{tikzpicture}[scale=.19]\splits{1,1,1,1,1,1,1,1,1,1,1,1,1,1,1,1,1,1,1,1,1,1,1,1,1,1,1,1}{\pho5}{5\pho}{\phoc}\end{tikzpicture}  
  \caption{Cake cutting for recipe $\sigma=1425367$ with two periods. 
  The ratios for portions of five pieces are shown at the right: 
  $\mu_5^\sigma=10/7$ is the greatest of these.}\label{figExample}
\end{figure}

Note that the specific order of the initial $p$ cuts is immaterial to our analysis, since reordering a finite prefix of $\ba$ has no effect on~$\mu_r(\ba)$.
Given a value of $r$ and a recipe $\sigma$, we use $\mu_r^\sigma$ to denote the value of $\mu_r(\ba_\sigma)$ where $\ba_\sigma$ is the sequence of cuts defined by~$\sigma$.

For any given value of $m$ (not necessarily a power of two), we refer to the $mp$ cuts defined by Step 2 as the \emph{phase with $m$ periods}.
See Figure~\ref{figExample} for an illustration (shown linearly) of the phase with two periods for recipe 1425367.

It turns out that we can determine $\mu_r^\sigma$ just by analysing a single phase, as long as it has sufficiently many periods.

\begin{propO}\label{propLen}
  Suppose $\sigma$ 
  is a recipe of 
  length $p$
  and $r\geqs1$.
  If $m\geqs2\ceil{r/p}$,
  then $\mu^\sigma_r$ equals the greatest ratio that occurs during the phase with $m$ periods.
\end{propO}
\begin{proof}
It is sufficient to prove that any phase with at least $2\ceil{r/p}$ periods yields the same set of ratios.

As noted above, after each multiple of $m$ cuts all the periods are identical. 
If each period consists of $\ell$ pieces ($p\leqs\ell\leqs2p$), then there are at most $\ell$ distinct portions. 
These are the same for any value of~$m$.

More generally, if the number of cuts made so far is congruent to $k$ modulo $m$ then the first $k$ (\emph{early}) periods are identical to each other, and the remaining $m-k$ (\emph{late}) periods are also identical to each other.
Moreover, an early period differs from a late period just by having one large piece replaced by an adjacent pair of small pieces, which we call a \emph{new pair}.
Without loss of generality, we may assume that new pairs occur at the beginning of the periods.

Now a period consists of at least $p$ pieces. 
So a portion can contain pieces from at most $q\eq\ceil{r/p}$ new pairs, even if there are more than $q$ early periods.

Suppose there are $k\leqs q$ early periods, each consisting of $\ell+1$ pieces.
Then there are at most $k(\ell+1)+r$ distinct portions, as illustrated in Figure~\ref{figProof}. (Note that the each blue part of a period is the same.)

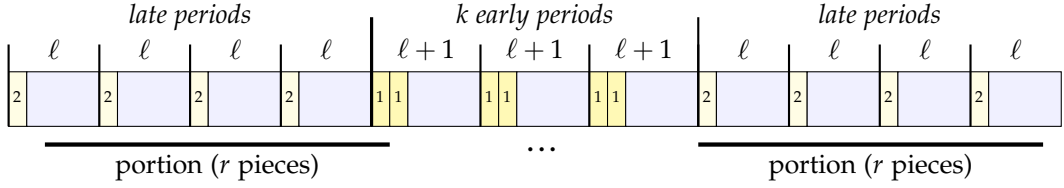
\begin{figure}[ht]
  \centering
\newcommand{\yheight}{3}
\small
\begin{tikzpicture}[scale=.24]
  \foreach \x in {-4,...,-1}
    {
      \draw[fill=yellow!12] (5*\x,0) rectangle (5*\x+1,\yheight);
      \node at (5*\x+.55,\yheight/2) {${}_{{}^2}$};
      \draw[fill=blue!6] (5*\x+1,0) rectangle (5*\x+5,\yheight);
      \draw[thick] (5*\x,0)--(5*\x,\yheight+1.5);
      \node at (5*\x+2.5,\yheight+1.25) {$\ell$};
    }
    \node at (5*-2.5+2.5,\yheight+3) {\emph{late periods}};
  \foreach \x in {0,...,2}
    {
      \draw[fill=yellow!36] (6*\x,0) rectangle (6*\x+1,\yheight);
      \node at (6*\x+.55,\yheight/2) {${}_{{}^1}$};
      \draw[fill=yellow!36] (6*\x+1,0) rectangle (6*\x+2,\yheight);
      \node at (6*\x+1+.55,\yheight/2) {${}_{{}^1}$};
      \draw[fill=blue!6] (6*\x+2,0) rectangle (6*\x+6,\yheight);
      \draw[thick] (6*\x,0)--(6*\x,\yheight+1.5);
      \node at (6*\x+3,\yheight+1.25) {$\ell+1$};
    }
  \draw[very thick] (0,0)--(0,\yheight+3);
  \node at (6*1+3,\yheight+3) {\emph{$k$ early periods}};
  \foreach \x in {3,...,6}
    {
      \draw[fill=yellow!12] (5*\x+3,0) rectangle (5*\x+4,\yheight);
      \node at (5*\x+3+.55,\yheight/2) {${}_{{}^2}$};
      \draw[fill=blue!6] (5*\x+4,0) rectangle (5*\x+8,\yheight);
      \draw[thick] (5*\x+3,0)--(5*\x+3,\yheight+1.5);
      \node at (5*\x+5.5,\yheight+1.25) {$\ell$};
    }
    \draw[thick] (18,0)--(18,\yheight+3);
    \node at (5*4.5+3+2.5,\yheight+3) {\emph{late periods}};
    \draw[ultra thick] (1-19,-1)--(1,-1);
    \node at (1-9.5,-2.25) {portion ($r$ pieces)};
    \node at (9.5,-1.25) {\textbf{\tel}};
    \draw[ultra thick] (18,-1)--(18+19,-1);
    \node at (18+9.5,-2.25) {portion ($r$ pieces)};
\end{tikzpicture}
  \caption{The leftmost and rightmost positions, with respect to the early periods, of the $k(\ell+1)+r$ possibly distinct portions; here $k=3$, $\ell=5$ and $r=19$.}\label{figProof}
\end{figure}

To ensure that all of these can be instantiated, it is sufficient that there are at least $\ceil{r/\ell}\leqs q$ late periods.
Hence $m\geqs2q$ suffices to capture all the possible ratios when there are at most $q$ early periods.

The creation of more than $q$ early periods does not result in any new possibilities for the structure of a portion, since a portion can only contain pieces from $q$ new pairs.
Thus every possible ratio occurs when there are $q$ or fewer early periods.
The result follows.
\end{proof}

Therefore, if $\sigma$ has length $p$, then to calculate $\mu_r^\sigma$ we only need to determine the greatest ratio that occurs during a phase with $2\ceil{r/p}$ periods.
In particular, if $p\geqs r$, then only two periods are required.
For example, $\mu_5^{1425367}=10/7$, as illustrated in Figure~\ref{figExample}.

Figure~\ref{figCode} contains some simple \emph{Mathematica} code for calculating $\mu_r^\sigma$.
(This could easily be optimised in a variety of ways.)

\begin{figure}[ht]
\small
\hrule
\vspace{6pt}
\setstretch{1.1}
\begin{verbatim}
initLength[r_, p_] := 2 Ceiling[r/p] p
splitPiece[a_, k_] := ReplacePart[a, k -> {1, 1}]
expandRecipe[s_, len_] := Join @@ (Range[#, len, Length@s]& /@ s)
maxMinRatio[r_][a_] := Module[{sizes = Total /@ Partition[Flatten@a, r, 1, {1, 1}]}, 
  Max@sizes / Min@sizes]
largestRatio[r_, s_] := Module[{len = initLength[r, Length@s]},
  Max[maxMinRatio[r] /@ 
    FoldList[splitPiece, ConstantArray[2, len], expandRecipe[s, len]]]]
\end{verbatim}
\vspace{-6pt}
\hrule
\vspace{6pt}
\caption{\emph{Mathematica} code: \texttt{\small largestRatio[}$r$\texttt{\small,}$\sigma$\texttt{\small ]} gives $\mu_r^\sigma$.}\label{figCode}
\end{figure}

With rather more sophisticated code, it is possible to search effectively for recipes that yield small ratios.
Table~\ref{tblResults} lists results for $r\leqs20$ with periods of length at most~24.
For each $r$, the smallest value of some $\mu_r^\sigma$ found is given, along with the shortest period and lexicographically earliest recipe that yields this value.
It seems likely that some of these might be improved by using longer periods.

\begin{table}[t]
  \centering\small
  \renewcommand*{\arraystretch}{1.225}
  \begin{tabular}{c|l|c|l}
    $r$ & \pho\hfill$\mu_r^\sigma$\hfill\pho     & \!\emph{period}\!    & \emph{recipe $\sigma$} \\\hline
     2  & $\pho2/1\pho = 2$             & \pho1     & 1 \\
     3  & $\pho3/2\pho = 1.5$           & \pho2     & 1\,2 \\
     4  & $\pho3/2\pho = 1.5$           & \pho3     & 1\,2\,3 \\
     5  & $\pho7/5\pho = 1.4$           & \pho4     & 1\,2\,3\,4 \\
     6  & $\pho4/3\pho \approx 1.33333$ & \pho4     & 1\,2\,3\,4 \\
     7  & $\pho5/4\pho = 1.25$          &    12     & 1\,7\,3\,9\,5\,11\,2\,6\,10\,4\,8\,12 \\
     8  & $\pho5/4\pho = 1.25$          & \pho6     & 1\,2\,3\,5\,4\,6 \\
     9  & $11/9\pho \approx 1.22222$    & \pho7     & 1\,2\,4\,6\,3\,5\,7 \\
    10  & $\pho6/5\pho = 1.2$           & \pho8     & 1\,3\,5\,7\,2\,6\,4\,8 \\
    11  & $\pho6/5\pho = 1.2$           &    20     & 1\,11\,3\,13\,5\,16\,8\,18\,7\,14\,20\,9\,2\,6\,12\,17\,4\,10\,15\,19 \\
    12  & $19/16 = 1.1875$              &    20     & 1\,11\,3\,13\,5\,9\,17\,7\,15\,19\,2\,8\,14\,4\,10\,16\,6\,12\,18\,20 \\
    13  & $20/17 \approx 1.17647$       &    22     & 1\,12\,3\,14\,5\,16\,7\,10\,19\,20\,4\,11\,18\,8\,21\,2\,9\,15\,6\,13\,17\,22 \\
    14  & $\pho7/6\pho \approx 1.16667$ &    23     & 1\,3\,13\,15\,5\,17\,7\,10\,20\,9\,18\,22\,2\,11\,4\,12\,19\,16\,6\,8\,14\,21\,23 \\ 
    15  & $15/13 \approx 1.15385$       &    24     & 1\,11\,13\,23\,3\,15\,5\,9\,19\,7\,17\,21\,2\,10\,18\,4\,12\,20\,6\,14\,22\,8\,16\,24 \\
    16  & $\pho8/7\pho \approx 1.14286$ &    21     & 1\,7\,14\,10\,19\,4\,17\,8\,11\,20\,2\,13\,5\,16\,3\,12\,18\,6\,9\,15\,21 \\
    17  & $\pho8/7\pho \approx 1.14286$ &    19     & 1\,5\,10\,14\,7\,17\,3\,12\,8\,15\,18\,2\,11\,4\,6\,13\,9\,16\,19 \\
    18  & $25/22 \approx 1.13636$       &    20     & 1\,5\,11\,15\,8\,18\,3\,13\,6\,16\,9\,19\,2\,12\,4\,14\,7\,10\,17\,20 \\
    19  & $26/23 \approx 1.13043$	    &    22     & 1\,6\,12\,17\,9\,20\,3\,14\,2\,13\,4\,15\,5\,16\,7\,18\,8\,11\,21\,10\,19\,22 \\
    20  & $\pho9/8\pho = 1.125$         &    24     & 1\,6\,13\,18\,9\,21\,3\,15\,10\,22\,2\,14\,4\,16\,5\,17\,7\,19\,11\,23\,8\,12\,20\,24
  \end{tabular}
  \caption{Upper bounds on $\mu_r$ for small $r$}\label{tblResults}
\end{table}

So far, the value of $\mu_r$ is only known for $r=1$.
Lower and upper bounds for $r\geqs2$ are given in Table~\ref{tblBounds}. 
Some natural questions include:
\begin{bullets}
  \item For which (if any) $r$ are these upper bounds tight? 
  \item In particular, what is the exact value of $\mu_2$? And does either $\mu_3=\frac32$ or $\mu_4=\frac32$?
  \item Are there other natural sequences of cuts that improve on these bounds?
  \item Does $\mu_r$ weakly decrease with $r$?
\end{bullets}

\section{The \vdc{} sequence}\label{sectVDC}

The (binary) \vdc{} sequence~\cite{Corput1935} is an infinite sequence $\bv=(v_i)_{i=1}^\infty$ of values in the unit interval, where $v_i$ is determined from the reverse of the binary expansion of $i-1$.
Here are the first eight values:
\[
0.0_2\!=\!0, \;\;\;
0.1_2\!=\!\tfrac12, \;\;\;
0.01_2\!=\!\tfrac14, \;\;\;
0.11_2\!=\!\tfrac34, \;\;\;
0.001_2\!=\!\tfrac18, \;\;\;
0.101_2\!=\!\tfrac58, \;\;\;
0.011_2\!=\!\tfrac38, \;\;\;
0.111_2\!=\!\tfrac78  . 
\]
If we consider our cake to have unit size, then the \vdc{} sequence specifies a sequence of radial cuts.
We claim that, for each $r$, we have $\mu_r(\bv)=\mu_r^\sigma$ for some recipe $\sigma$ that depends on~$r$.

For each $h\geqs0$, let $\vphi^h=\vphi^h_1,\ldots,\vphi^h_{2^h}$ be the recipe of length $2^h$ defined from the first $2^h$ terms in the \vdc{} sequence by $\vphi^h_i=2^hv_i+1$. 
For example, we have $\vphi^3=15372648$.
Then we have the following identity:

\begin{propO}\label{propVdC}
  For each $r\geqs1$ and $h\geqs\log_2r$, we have $\mu_r(\bv)=\mu_r^{\vphi^h}$.
\end{propO}


\begin{proof}
  We begin by observing the similarities between the \vdc{} sequence $\bv$ and the sequence with recipe $\vphi^h$ (for some~$h$).
  
  Note first that, for any $j$, the first $2^j$ cuts of $\bv$ divide the cake into $2^j$ equally-sized (\emph{large}) pieces.
  Moreover, the subsequent $2^j$ cuts (which we'll refer to as \emph{phase $j$} of $\bv$) each bisect one of these pieces (the \emph{old piece}), creating an adjacent \emph{new pair} of \emph{small} pieces.
  Furthermore, for any $h<j$, if we consider the $2^j$ pieces to consist of $m=2^{j-h}$ periods (each consisting of $p=2^h$ consecutive pieces), then after each multiple of $m$ cuts, all the periods are identical.
  
  More generally, if the number of cuts made so far in the phase is congruent to $k$ modulo $m$ then there are $k$ \emph{early} periods that are identical to each other, and $m-k$ \emph{late} periods which are also identical to each other, with an early period differing from a late period just by having one large piece replaced by a new pair.
  
  Now, due to the recursive structure of $\bv$, the order in which the large pieces in each period are bisected is given by $\vphi^h$.
  So phase~$j$ of the \vdc{} sequence (for any $j>h$) behaves in a similar manner to the sequence with recipe $\vphi^h$, but with the difference that the early periods are usually not consecutive.
  (Indeed, the ordering of the early periods is given by $\vphi^{j-h}$, so, for example, the first $m/2$ early periods are the odd-numbered periods.)
  
  Suppose now that $h\geqs\log_2r$, or equivalently that $p\geqs r$.
  By Proposition~\ref{propLen}, it is sufficient to demonstrate that the phase \emph{with two periods} for recipe $\vphi^h$ (which we call the \emph{$\vphi^h$ phase} for brevity) yields the same set of ratios (for any $j>h$) as phase~$j$ of $\bv$ (which we now refer to as the \emph{$\bv$ phase}).
  
  Firstly, as noted above, for any~$s$, after $sm$ cuts in the $\bv$ phase, all the periods are identical.
  Similarly, after $2s$ cuts in the $\vphi^h$ phase, both periods are the same, and are the same as those in the $\bv$ phase after $sm$ cuts.
  Thus we have the same set of portions in the $\vphi^h$ phase as in the $\bv$ phase.
  
  Now, after $sm+k$ cuts ($1\leqs k<m$) in the $\bv$ phase there are $k$ early periods and $m-k$ late periods.
  Similarly, after $2s+1$ cuts in the $\vphi^h$ phase there is one early period and one late period, matching the two distinct periods in the $\bv$ phase.
  We claim that we have the same set of portions in the two phases (for any choice of~$k$).
  
  Since $r\leqs p$, a portion can contain pieces from at most two (adjacent) periods.
  Moreover, a portion can contain only one old piece (in a late period), and can contain pieces from only one new pair (in an early period).
  It cannot contain both an old piece and also a piece from a new pair.
  
  There are thus only three possibilities for a portion:
  \begin{bulletnums}
    \item The portion doesn't contain an old piece or a piece from a new pair. 
    In this case, the same set of portions is available whether the pieces in such a portion come from early or late periods.
    \item The portion contains an old piece in a late period. 
    Since such a portion cannot contain another old piece or a piece from a new pair, 
    the same set of portions is available whether the periods adjacent to the late period are early or late.
    \item The portion contains one or two pieces from a new pair in an early period. 
    Since such a portion cannot contain an old piece or a piece from another new pair, 
    again the same set of portions is available whether the periods adjacent to the early period are early or late.
  \end{bulletnums}  
  
  The three possibilities are illustrated in Figure~\ref{figProofVdeC}, where without loss of generality it is assumed that old pieces and new pairs occur at the beginning of the periods. 
  (Note again that the each blue part of a period is the same.)
  
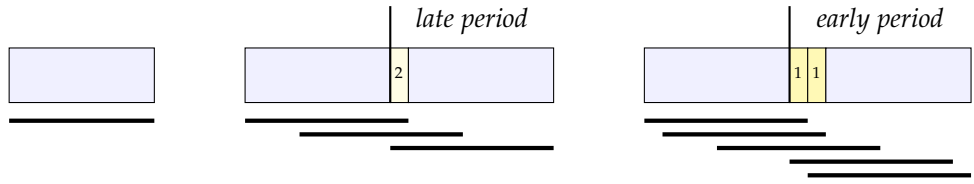
\begin{figure}[t]
  \centering
\newcommand{\yheight}{3}
\small
\begin{tikzpicture}[scale=.24]
    \draw[fill=blue!6] (-30,0) rectangle (-22,\yheight);
    \draw[ultra thick] (-30,-1)--(-22,-1);
    \draw[fill=blue!6] (-17,0) rectangle (-9,\yheight);
    \draw[fill=yellow!12] (-9,0) rectangle (-8,\yheight);
    \node at (-9+.55,\yheight/2) {${}_{{}^2}$};
    \draw[fill=blue!6] (-8,0) rectangle (0,\yheight);
    \draw[thick] (-9,0)--(-9,\yheight+2.3);
    \node at (-4.5,\yheight+1.5) {\emph{late period}};
    \draw[ultra thick] (-17,-1)--(-8,-1);
    \draw[ultra thick] (-14,-1.75)--(-5,-1.75);
    \draw[ultra thick] (-9,-2.5)--(0,-2.5);
    \draw[fill=blue!6] (5,0) rectangle (13,\yheight);
    \draw[fill=yellow!36] (13,0) rectangle (14,\yheight);
    \node at (13.55,\yheight/2) {${}_{{}^1}$};
    \draw[fill=yellow!36] (14,0) rectangle (15,\yheight);
    \node at (14+.55,\yheight/2) {${}_{{}^1}$};
    \draw[fill=blue!6] (15,0) rectangle (23,\yheight);
    \draw[thick] (13,0)--(13,\yheight+2.3);
    \node at (18,\yheight+1.5) {\emph{early period}};
    \draw[ultra thick] (5,-1)--(14,-1);
    \draw[ultra thick] (6,-1.75)--(15,-1.75);
    \draw[ultra thick] (9,-2.5)--(18,-2.5);
    \draw[ultra thick] (13,-3.25)--(22,-3.25);
    \draw[ultra thick] (14,-4)--(23,-4);
\end{tikzpicture}
  \caption{Some possible portions, in the extremal case when a portion has the same number of pieces as an early period}\label{figProofVdeC}
\end{figure}  
  
  Since each of the possible portions is present irrespective of the order in which early and late periods occur (the only requirement being the presence of at least one of each), we have the same set of portions in the $\vphi^h$
phase as in the $\bv$ phase.
\end{proof}

As a consequence of Proposition~\ref{propVdC}, it is easy to calculate $\mu_r(\bv)$, as presented above. 
The first few values are shown in Table~\ref{tblVdeC}.

\begin{table}[ht]
\small
\[
\renewcommand*{\arraystretch}{1.2}
\begin{array}{c||c|c|c|c|c|c|c|c|c|c|c|}
r&2&3&4&5&6&7&8&9&10&11&12 \\\hline
\multirow{2}{*}{$\mu_r(\bv)$}& \multirow{2}{*}{2} & {3/2} & {3/2} & {7/5} & {4/3} & {7/5} & {7/5} & {5/4} & {6/5} & 11/9 & 11/9 \\
&  & 1.5 & 1.5 & 1.4 & 1.33333 & 1.4 & 1.4 & 1.25 & 1.2 & 1.22222 & 1.22222
\end{array}
\]
  \caption{Asymptotic ratios for the binary \vdc{} sequence}\label{tblVdeC}
\end{table}

Note that $\mu_r(\bv)$ does not weakly decrease with $r$.
Comparison with our upper bounds for $\mu_r$ in Table~\ref{tblResults} shows that $\mu_r<\mu_r(\bv)$ for $r=11,\ldots,20$.
Is it the case that, for sufficiently large $r$, we always have $\mu_r<\mu_r(\bv)$?


\vspace{12pt}
\begin{flushright}
\emph{Soli Deo gloria!}
\end{flushright}

\bibliographystyle{plain}
{\footnotesize\bibliography{../bib/mybib}}

\end{document}